\documentclass[11pt, oneside]{article}   	% use "amsart" instead of "article" for AMSLaTeX format
\usepackage{geometry}                		% See geometry.pdf to learn the layout options. There are lots.
\usepackage{graphicx}				% Use pdf, png, jpg, or eps§ with pdflatex; use eps in DVI mode
\usepackage[all]{xy}
\CompileMatrices

\usepackage{amssymb}
\usepackage{amsmath}
\usepackage{stmaryrd}
\usepackage{amsthm}
\usepackage{tikz-cd}
\usepackage{extarrows}
\usepackage[mathscr]{euscript}
\usepackage{mathrsfs} 
\usepackage{verbatim}
\usepackage[OT2,T1]{fontenc}
\DeclareSymbolFont{cyrletters}{OT2}{wncyr}{m}{n}
\DeclareMathSymbol{\Sha}{\mathalpha}{cyrletters}{"58}
\DeclareMathSymbol{\Che}{\mathalpha}{cyrletters}{"51}

\usepackage{calligra}
\usepackage{mathrsfs}

\newcommand{\Ga}{{\mathbf{G}}_{\rm{a}}}
\newcommand{\Gm}{{\mathbf{G}}_{\rm{m}}}

\DeclareMathOperator{\Pic}{Pic}
\DeclareMathOperator{\Jac}{Jac}

\DeclareMathOperator{\R}{R}

\newcommand*{\Z}{\ensuremath{\mathbf{Z}}}                        % integers
\newcommand*{\F}{\ensuremath{\mathbf{F}}}                        % field
\newcommand*{\A}{\ensuremath{\mathbf{A}}}                        % affine/adele
\renewcommand*{\P}{\ensuremath{\mathbf{P}}}                        % proj space
\newcommand*{\calO}{\mathcal{O}}                                  % 'sheaf' O
\usepackage{rotating}
\newcommand*{\isoarrow}[1]{\arrow[#1,"\rotatebox{90}{\(\sim\)}"
]}
\usepackage{bm}

\numberwithin{equation}{section}

\newtheorem{theorem}{Theorem}[section]

\newtheorem{lemma}[theorem]{Lemma}
\newtheorem{proposition}[theorem]{Proposition}
\newtheorem{corollary}[theorem]{Corollary}

\theoremstyle{definition}
  \newtheorem{definition}[theorem]{Definition}

\theoremstyle{remark}
  \newtheorem{remark}[theorem]{Remark}

\theoremstyle{definition}
  \newtheorem{example}[theorem]{Example}

\usepackage[OT2,T1]{fontenc}

\tikzset{commutative diagrams/.cd,
mysymbol/.style={start anchor=center,end anchor=center,draw=none}
}

\usepackage{stmaryrd}
\usepackage{hyperref}

\title{\textbf{(NON-)EXTENDABILITY OF ABEL-JACOBI MAPS}}
\author{Zev Rosengarten \thanks{
While completing this work, the author was supported by Israel Science Foundation Grant No.\,2083/24.
}}

\date{}
\begin{document}
\maketitle

\begin{abstract}
We investigate the ``natural'' locus of definition of Abel-Jacobi maps. In particular, we show that, for a proper, geometrically reduced curve $C$ -- not necessarily smooth -- the Abel-Jacobi map from the smooth locus $C^{\rm{sm}}$ into the Jacobian of $C$ does not extend to any larger (separated, geometrically reduced) curve containing $C^{\rm{sm}}$ except under certain particular circumstances which we describe explicitly. As a consequence, we deduce that the Abel-Jacobi map has closed image except in certain explicitly described circumstances, and that it is always a closed embedding for irreducible curves not isomorphic to $\P^1$.
\end{abstract}

\setcounter{tocdepth}{1}
\tableofcontents{}

% Section: Introduction
\section{Introduction}

Given a projective curve $C$ over a field $K$, perhaps the most useful tool in the study of $C$ is the Abel-Jacobi map. Let us recall its definition. Let $\Pic_{C/K}$ denote the Picard functor of $C$. It is (represented by) a smooth $K$-scheme \cite[\S8.2 Th.\,3, \S8.4, Prop.\,2]{neronmodels}.

% Definition: Jacobian
\begin{definition}
\label{jacobiandef}
For a projective curve $C$ over a field $K$, the {\em Jacobian} of $C$, denoted $\Jac(C)$, is the identity component $\Pic^0_{C/K}$ of the Picard scheme of $C$.
\end{definition}

Now let $C$ be an irreducible projective curve over the field $K$, and let $C^{\rm{sm}}$ denote the smooth locus of $C$. If we fix a point $c_0 \in C^{\rm{sm}}(K)$, then the Abel-Jacobi map (with basepoint $c_0$) is the map
\[
\iota_{c_0}\colon C^{\rm{sm}} \rightarrow \Jac(C), \hspace{.3 in} c \mapsto \calO([c]-[c_0]).
\]
Such maps are most familiar when $C$ is smooth. But even if one is mainly interested in smooth curves, nonsmooth curves and their Abel-Jacobi maps arise naturally. For instance, consider the following situation which often arises in arithmetic applications. Let $R$ be a DVR with fraction field $K$, and suppose given a smooth connected projective curve $C/K$. Then $C$ extends to a flat projective curve $\mathscr{C}/R$. Unfortunately, the special fiber $C_0$ of $\mathscr{C}$ may fail to be smooth. Nevertheless, the Abel-Jacobi map associated to $C_0$ plays an important role in understanding the arithmetic of $C$.

A natural question concerning the Abel-Jacobi map is its maximal locus of definition. More precisely, one may ask whether the map may be extended to a map from a larger (separated) curve (not necessarily contained in $C$) containing $C^{\rm{sm}}$ as a dense open subscheme. It is not hard to check that the answer is no when $C^{\rm{sm}} = C$; see Example \ref{smpropexample} below. But in general, this is not at all clear. The purpose of the present paper is to investigate this question. In particular, we will prove that $C^{\rm{sm}}$ really is in a precise sense the maximal locus of definition of the Abel-Jacobi map except in very particular circumstances, and in particular, it always is when $C$ is irreducible, and we will deduce various consequences from this result.

Before stating our result in full generality, we must define the Abel-Jacobi map beyond the irreducible setting. So let $C/K$ be a projective curve, and let $C_1, \dots, C_n$ be the irreducible components of $C$. For each $i$, suppose given $c_i \in (C_i \cap C^{\rm{sm}})(K)$. Let $\vec{c} := \{c_1, \dots, c_n\}$. The Abel-Jacobi map
\[
\iota_{\vec{c}}\colon C^{\rm{sm}} \rightarrow \Jac(C)
\]
is the map sending $c \in C_i \cap C^{\rm{sm}}$ to the line bundle $\calO([c]-[c_i])$. More precisely, $\iota_{\vec{c}}$ is the map which on $X_i := C_i \cap C^{\rm{sm}}$ corresponds to the line bundle $\calO([\Delta|_{C \times X_i}] - [c_i \times X_i])$ on $C \times X_i$, where $\Delta \subset C \times C$ is the diagonal. This makes sense because the $X_i$ are disjoint open subschemes of $C$. The map $\iota_{\vec{c}}$ does indeed land in $\Jac(C)$ because each $X_i$ is connected.

Before turning to our results, we first note that there are certain assumptions that must be imposed if the Abel-Jacobi map is not to extend beyond $C^{\rm{sm}}$. First of all, we must impose a reducedness assumption on $C$. To see why, if $C$ is a generically reduced curve, then a similar argument as in the proof of \cite[Lem.\,2.2.9]{rostateduality} (with $i = 1$ and $G = \Gm$) shows that the map $\Pic(C) \rightarrow \Pic(C_{\rm{red}})$ is an isomorphism. This argument can be upgraded to show that the map $\Jac(C) \rightarrow \Jac(C_{\rm{red}})$ is an isomorphism. Thus if, say, $C$ is generically smooth but not everywhere reduced, with $C_{\rm{red}}$ smooth, then we see that the Abel-Jacobi map into $\Jac(C) = \Jac(C_{\rm{red}})$ extends from $C^{\rm{sm}}$ to the larger curve $C_{\rm{red}} \supset C^{\rm{sm}}$.

An additional necessary assumption is that we must restrict attention to {\em separated} curves containing $C^{\rm{sm}}$. To see what can go wrong without this restriction, let $C$ be a, say, smooth proper curve, and let $C'$ be $C$ with a doubled point, a non-separated curve. That is, $C'$ is obtained by gluing two copies of $C$ (via the identity map) along the complements of a single closed point. Then the Abel-Jacobi map $C \rightarrow \Jac(C)$ extends to $C'$ (just take the Abel-Jacobi map along each of the glued copies of $C$).

Finally, there is one more, somewhat subtler issue, involving both local and global invariants of the curve $C$, that can arise in general if $C$ is not irreducible. 

% Definition: Modifiable
\begin{definition}
\label{modifiabledef}
Let $C$ be a geometrically reduced projective curve, let $X$ be an irreducible component of $C$, and let $Y$ be the (reduced) union of the other components. Assume that $x \in X^{\rm{sm}}(K)$ lies in $X \cap Y$, and that $X \cap Y \cap U = x$ scheme-theoretically for some open neighborhood $U \subset C$ of $x$. Let $C'$ denote the curve obtained by ``pulling $X$ away from $Y$ along $x$.'' That is, glue $X$ and $Y$ along $(X \cap Y)-\{x\}$. To be slightly more precise, let $Z := (X \cap Y)-\{x\} \subset C$. Then $C'$ is obtained by glueing $C-\{x\}$ and $(X-Z) \coprod (Y-Z)$ along their common open subset $(X - (X \cap Y)) \coprod (Y - (X \cap Y))$.

One then has a natural map $f\colon C' \rightarrow C$ which is an isomorphism above the complement of $x$. Additionally assume that $f$ is not injective on connected components. In this situation, we say that $C$ is {\em modifiable}, and that $C'$ is a {\em modification} of $C$, or is the {\em modification of $C$ along $(X, x)$}.
\end{definition}

Note that irreducible curves are never modifiable. We will show that, if $C$ is modifiable as above, then the pullback map $f^*\colon\Jac(C) \rightarrow \Jac(C')$ is an isomorphism (Proposition \ref{modjacobian}). In particular, the Abel-Jacobi map for $C$ extends over the strictly larger smooth locus of $C'$, which contains the new point on $X$ mapping to $x$.

The first main result of the present paper is to prove that the above issues are the only manners in which the Abel-Jacobi map may extend beyond the smooth locus.

% Theorem: Non-extendability of Abel-Jacobi maps
\begin{theorem}
\label{nonextaj}
Let $C/K$ be a geometrically reduced projective curve, and let $\vec{c} := \{c_1, \dots, c_n\}$ be a collection of rational points in $C^{\rm{sm}}$, one in each irreducible component. Assume that $C$ is not modifiable. Then for any dominant open embedding $C^{\rm{sm}} \subsetneq C'$ with $C'/K$ a separated curve, the Abel-Jacobi map $\iota_{\vec{c}}\colon C^{\rm{sm}} \rightarrow \Jac(C)$ does not extend to a map $C' \rightarrow \Jac(C)$.
\end{theorem}

We quickly deduce the following corollary. 

% Corollary: C^{sm} is Zariski closed in the Jacobian
\begin{corollary}
\label{zarclinjac}
Let $C/K$ be a geometrically reduced projective curve, and let $\vec{c} := \{c_1, \dots, c_n\}$ be a collection of rational points in $C^{\rm{sm}}$, one in each irreducible component. Assume that $C$ is not modifiable. Then the Abel-Jacobi map $\iota_{\vec{c}}\colon C^{\rm{sm}} \rightarrow \Jac(C)$ has closed image.
\end{corollary}

\begin{proof}
Let $C' \subset \Jac(C)$ be the Zariski closure of $\iota_{\vec{c}}(C^{\rm{sm}})$. The scheme $\Jac(C)$, being a locally finite type group scheme over a field, is separated, hence so is $C'$. Because the Abel-Jacobi map extends to $C'$, one has $C' = \iota_{\vec{c}}(C^{\rm{sm}})$ by Theorem \ref{nonextaj}.
\end{proof}

\begin{remark}
In fact, even for modifiable $C$ the image of the Abel-Jacobi map may be closed. For instance, if we take $C$ to be the union of two copies of $\P^1$ meeting at an ordinary double point, then $\Jac(C) = \Jac(C')$, where $C'$ is the disjoint union of two copies of $\P^1$ (Proposition \ref{modjacobian}). Thus $\Jac(C) = 0$, so clearly the Abel-Jacobi map has closed image in this case. The right condition for the Abel-Jacobi map to have closed image is a tad complicated. The requirement is that $C$ only be modifiable by $\P^1$'s. To explain what we mean by this, a modification of $C$ by $\P^1$ is a modification at a point $x$ such that either $x$ lies in exactly two components of $C$, both isomorphic to $\P^1$, or else $x$ lies in more than two components, and the component by which one is modifying (the $X$ in the definition of modifiability) is isomorphic to $\P^1$. We say that $C$ is only modifiable by $\P^1$'s when there is a series of modifications of $C$ by $\P^1$ ending in a non-modifiable curve. We will never use this.
\end{remark}

Because irreducible curves are never modifiable, Theorem \ref{nonextaj} and Corollary \ref{zarclinjac} apply in particular to such curves. We will also prove the following theorem, which says that, if we additionally assume that $C$ is irreducible, then the Abel-Jacobi map is a closed embedding except when $C \simeq \P^1$.

% Theorem: Abel-Jacobi map is a closed embedding when C is irreducible
\begin{theorem}
\label{ajclemb}
Let $C/K$ be an irreducible, geometrically reduced projective curve not isomorphic to $\P^1_K$, and let $c_0 \in C^{\mathrm{sm}}(K)$. Then the Abel-Jacobi map $\iota_{c_0}\colon C^{\rm{sm}} \rightarrow \Jac(C)$ is a closed embedding.
\end{theorem}

We now give some examples to illustrate the above theorems.

\begin{example}
\label{smpropexample}
If $C$ is a smooth projective curve, then any dominant open embedding $C \hookrightarrow C'$ into a separated curve is an isomorphism, so Theorem \ref{nonextaj} is trivial in this case.
\end{example}

\begin{example}
Let $C$ be the nodal cubic $Y^2 = XY + X^3$. Then the normalization $\pi\colon \widetilde{C} \rightarrow C$ is the projective line. If $\mathscr{F}$ is the sheaf defined by the following exact sequence
\[
0 \longrightarrow \calO_C^{\times} \longrightarrow \pi_*(\calO_{\widetilde{C}}^{\times}) \longrightarrow \mathscr{F} \longrightarrow 0
\]
then one readily computes that ${\rm{H}}^0(C, \mathscr{F}) = (K^{\times} \times K^{\times})/\Delta(K^{\times}) \simeq K^{\times}$, where $\Delta$ is the diagonal map. Using the inclusion ${\rm{H}}^1(C, \pi_*(\calO_{\widetilde{C}}^{\times})) \hookrightarrow {\rm{H}}^1(\widetilde{C}, \calO_{\widetilde{C}}^{\times}) = \Pic(\widetilde{C}) = \Z$, one obtains an exact sequence
\[
0 \longrightarrow K^{\times} \longrightarrow \Pic(C) \longrightarrow \Z \longrightarrow 0.
\]
The above can be upgraded to an isomorphism
\[
\Jac(C) \simeq \Gm,
\]
and in this case the Abel-Jacobi map is an isomorphism. One can do a similar calculation with the cuspidal cubic $Y^2 = X^3$ and deduce that the Jacobian is $\Ga$ and that once again the Abel-Jacobi map is an isomorphism. These two examples will play an important role in the proof of Theorem \ref{ajclemb}, because -- over algebraically closed fields -- they cover every reduced, non-smooth, projective rational curve; see Lemma \ref{cubiccover}.
\end{example}

\begin{example}
\label{lutexample}
(due to Werner L\"utkebohmert) Let $C$ be the union of two copies of the projective line meeting at a pair of ordinary double points. Then the normalization $\pi\colon \widetilde{C} \rightarrow C$ is a disjoint union of two projective lines, and if $\mathscr{F}$ is defined by the exact sequence
\begin{equation}
\label{exampleeqn1}
0 \longrightarrow \calO_C^{\times} \longrightarrow \pi_*(\calO_{\widetilde{C}}^{\times}) \longrightarrow \mathscr{F} \longrightarrow 0,
\end{equation}
then one readily checks that one has an isomorphism ${\rm{H}}^0(C, \mathscr{F}) = [(K^{\times} \times K^{\times})/\Delta(K^{\times})]^2$, where $\Delta$ is the diagonal map, such that the map ${\rm{H}}^0(C, \pi_*(\calO_{\widetilde{C}})) = K^{\times} \times K^{\times} \rightarrow {\rm{H}}^0(C, \mathscr{F})$ is identified with the map induced by
\[
(\alpha, \beta) \mapsto (\alpha, \beta) \times (\alpha, \beta).
\]
Using the long exact sequence associated to (\ref{exampleeqn1}) and the inclusion ${\rm{H}}^1(C, \pi_*(\calO_{\widetilde{C}}^{\times}) \hookrightarrow {\rm{H}}^1(\widetilde{C}, \calO_{\widetilde{C}}^{\times}) = \Pic(\widetilde{C}) = \Z^2$, one obtains an exact sequence
\[
0 \longrightarrow K^{\times} \longrightarrow \Pic(C) \longrightarrow \Z^2 \longrightarrow 0.
\]
The above calculation can be upgraded to an isomorphism
\[
\Jac(C) \simeq \Gm,
\]
and in particular we see that the Abel-Jacobi map $C^{\rm{sm}} = (\A^1\backslash\{0\}) \coprod (\A^1\backslash\{0\}) \rightarrow \Jac(C)$ associated to some choice of rational points in each irreducible component of $C$ cannot be an embedding. Thus the irreducibility is essential in Theorem \ref{ajclemb}. However, the map still does have closed image, as asserted by Corollary \ref{zarclinjac}.
\end{example}

\begin{example}
Let $C$ be the union of two elliptic curves $E_1$, $E_2$ meeting at an ordinary double point which is not the identity of either curve. Then by Proposition \ref{modjacobian}, the normalization map $\widetilde{C} := E_1 \coprod E_2 \rightarrow C$ induces an isomorphism on Jacobians. Thus $\Jac(C) = \Jac(\widetilde{C}) = E_1 \times E_2$. Let $0 \neq p_i \in E_i(K)$ be the preimages of the node on $C$. Then the Abel-Jacobi map of $C$ corresponding to the two points $0_i \in E_i(K)$ is the map $C^{\rm{sm}} = (E_1-\{p_1\}) \coprod (E_2 - \{p_2\}) \rightarrow E_1 \times E_2$ which on $E_i-\{p_i\}$ is the inclusion into the $i$ factor and $0$ on the other factor. In particular, we see that the Abel-Jacobi map for $C$ has image $(0 \times (E_2-\{p_2\}))\cup ((E_1-\{p_1\}) \times 0)$, which is not closed in $E_1 \times E_2$. Thus the assumption in Corollary \ref{zarclinjac} is necessary.
\end{example}

\subsection*{Acknowledgements}
I wish to thank Werner L\"utkebohmert for providing Example \ref{lutexample}, for correcting an earlier version of Corollary \ref{zarclinjac}, and for other helpful suggestions. I would also like to thank the anonymous referee for his/her careful reading of the manuscript and many helpful comments and suggestions, including the proof of Lemma \ref{ajclembnodecusp} (and correcting an earlier incorrect proof).

\subsection{Notation and terminology}

Throughout this paper, $K$ denotes a field, and $K_s$, $\overline{K}$ denote separable and algebraic closures of $K$, respectively.

% Section: Preparations for the main results
\section{Preparations for the main results}

We begin with the following basic (probably well-known) lemma. Recall that a distinguished affine open subset of an affine scheme $X$ is one of the form $X_f$ for some $f \in \Gamma(X, \calO_X)$.

% Lemma: Existence of distinguished opens containing a given finite set
\begin{lemma}
\label{distopencontain}
Let $X$ be an affine scheme, and let $Z \subset X$ be a finite subset and $V \subset X$ an open subset containing $Z$. Then there is a cover of $V$ by distinguished affine open subsets of $X$ containing $Z$.
\end{lemma}

\begin{proof}
It is enough to show that there is a distinguished affine open of $X$ contained in $V$ and containing $Z$. Indeed, assuming this result for general $Z$, one may for each $v \in V$ apply the lemma with $Z$ replaced by $Z \cup\{v\}$ to obtain a distinguished open $U_v$ inside $V$ containing $Z\cup\{v\}$. The $U_v$ then provide the required cover of $V$.

Let $A := \Gamma(X, \calO_X)$, and let $Z = \{z_1 := \mathfrak{p}_1, \dots, z_n := \mathfrak{p}_n\}$. Choose for each $i$ a distinguished affine open $D(f_i) := {\rm{Spec}}(A_{f_i}) \subset V$ containing $z_i$. We may replace $V$ by $\cup_{i=1}^n D(f_i) = X-V(I)$, where $I := (f_1, \dots, f_n)$. We wish to find $g \in A$ such that $D(g) \cap V(I) = \emptyset$ but $Z \subset D(g)$. The former condition is equivalent to $g \in \sqrt{I}$, the radical of $I$, while the latter is equivalent to $g \notin \mathfrak{p}_1\cup\dots\cup \mathfrak{p}_n$. Thus we need to show that $\sqrt{I} \not\subset \mathfrak{p}_1\cup\dots\cup\mathfrak{p}_n$. By the prime avoidance lemma, this is equivalent to $\sqrt{I} \not\subset \mathfrak{p}_i$ for all $i$, and this in turn holds because $z_i \in D(f_i) \Longrightarrow f_i \not\in \mathfrak{p}_i$.
\end{proof}

In order to prove Theorem \ref{nonextaj}, we will have to be able to contract finite subschemes of a given scheme down to a rational point. In order to prove Theorem \ref{ajclemb}, we will also require the existence of a universal such contraction, and to be able to identify this universal contraction. We will only require the curve case, but since it is no more difficult, we carry out the construction in general. Such contractions are also constructed in \cite{ferrand} and \cite{schwede}, but we include the result below for the convenience of the reader.

% Proposition: Contracting finite subschemes
\begin{proposition}$($Contracting Finite Subschemes$)$
\label{contract}
Let $X$ be a quasi-projective $K$-scheme, and let $Z \subset X$ be a finite closed subscheme. Then there is a morphism $f\colon X \rightarrow X/Z$ of $K$-schemes and a point $y \in (X/Z)(K)$ such that $Z \subset f^{-1}(y)$ scheme-theoretically and such that one has the following universal property: Given a morphism of $K$-schemes $g\colon X \rightarrow W$ and a point $w \in W(K)$ such that $Z \subset g^{-1}(w)$ scheme-theoretically, there is a unique $K$-morphism $h\colon X/Z \rightarrow W$ such that $g = h\circ f$. Furthermore,
\begin{itemize}
\item[(1)] The triple $(f, X/Z, y)$ is unique up to unique isomorphism.
\item[(2)] For any open subscheme $V \subset X$ containing $Z$, the natural map $V/Z \rightarrow X/Z$ is an open embedding.
\item[(3)] $f$ induces an isomorphism $X-Z \xrightarrow{\sim} (X/Z)-\{y\}$.
\item[(4)] $f^{-1}(y) = Z$ scheme-theoretically.
\item[(5)] $f$ has schematically dense image.
\item[(6)] $f$ is a finite morphism, and $X/Z$ is a finite type $K$-scheme.
\item[(7)] If $X = {\rm{Spec}}(A)$ and $Z = V(I)$, then $X/Z = {\rm{Spec}}(K[I])$, where $K[I] := \{\lambda + i\mid \lambda \in K, i \in I\}$, $y$ is the map $\phi\colon K[I] \rightarrow K$, $\lambda + i \mapsto \lambda$, and $f$ is induced by the inclusion $K[I] \subset A$.
\end{itemize}
\end{proposition}

\begin{proof}
Once one has constructed $f$ with the universal property, condition (1) follows from Yoneda's Lemma. We first construct $f$ when $X = {\rm{Spec}}(A)$ is affine via the construction in (7). Let us check that $(f, X/Z, y)$ has the required universal property with respect to maps into {\em affine} $K$-schemes. Thus let $\zeta\colon B \rightarrow A$ be a $K$-algebra homomorphism, and suppose given a $K$-homomorphism $\psi\colon B \rightarrow K$ with $\zeta(\ker(\psi)) \subset I$. Then we must check that $\zeta$ lands in $K[I] \subset A$. To see this, let $b \in B$. Then we may write $b = \lambda + \alpha$ with $\lambda \in K$ and $\alpha \in \ker(\psi)$, hence $\zeta(b) = \lambda + \zeta(\alpha) \in K[I]$, as required.

Next we check the properties (3)--(5). One has $f^{-1}(y) = Z$, and $f$, being induced by an injective map, has schematically dense image. To see that $f$ induces an isomorphism above the complement of $y$, let $g \in K[I]$ with $y \notin {\rm{Spec}}(K[I]_g)$ -- that is, $\phi(g) = 0$, or equivalently, $g \in I$. We must show that the map $f^*_g\colon K[I]_g \rightarrow A_g$ is an isomorphism. Injectivity is immediate. For surjectivity, let $a \in A$. Then $ga \in I \subset K[I]$, so $a = f^*_g(ga/g)$. Thus $f^*_g$ is surjective.

Next we check (6). That is, we prove that $A$ is a finite $K[I]$-module, and that $K[I]$ is a finitely-generated $K$-algebra. The following proof is basically the one given by \cite{overflow}. We claim that $A$ is integral over $K[I]$. Indeed, fix an algebraic closure $\overline{K}$ of $K$, and let $a \in A$. Consider the element $F(T) := \prod_{z \in Z(\overline{K})} (T - a(z)) \in \overline{K}[T]$. Then $F(a)$ vanishes at every point of $Z_{\overline{K}} = V(I_{\overline{K}})$, hence $F(a)^n \in \overline{K}I$ for some $n > 0$. This gives an integral equation for $a$ over $\overline{K}[I] \subset \overline{K} \otimes_K A$. Thus $a$ is integral over $\overline{K}[I]$, which in turn is integral over $K[I]$, hence $a$ is integral over $K[I]$. Since $a \in A$ was arbitrary, $A$ is integral over $K[I]$. Because $A$ is a finitely-generated $K$-algebra, it follows that $A$ is finite over $K[I]$, and that $A$ is integral over $C$ for some finitely-generated $K$-subalgebra $C \subset K[I]$. Then $A$ is a finite $C$-module, hence so is $K[I]$. Thus $K[I]$ is a finitely-generated $K$-algebra.

Next suppose that $f \in A$ with $Z \subset D(f) \subset {\rm{Spec}}(A)$, where $D(f)$ is the distinguished affine open associated to $f$. Then we check that the natural map $D(f)/Z \rightarrow X/Z$ is an open embedding. This is automatic away from $y$ thanks to (3), hence only needs to be checked near $y = f(Z)$. In order to prove the assertion, therefore, one may replace $D(f)$ by $D(g)$ with $Z \subset D(g) \subset D(f)$. If $g \in K[I]$, then the map $K[I] \rightarrow K[I_g]$ becomes an isomorphism upon localizing at $g$, hence we would be done. Thus we wish to construct $g \in \sqrt{f} \cap K[I]$ such that $Z \subset D(g)$ under the assumption that $Z \subset D(f)$. We will construct $g \in \sqrt{f}\cap(1+I)$. The image $\overline{f}$ of $f$ in the finite $K$-algebra $A/I$ is integral over $K$; let $P(T) \in K[T]$ be the minimal polynomial. If the constant term of $P$ vanished, then -- because $f$ is invertible along $Z$ -- one could divide by $\overline{f}$ to obtain a smaller polynomial. Thus $P(0) \neq 0$. Multiplying through by $-P(0)^{-1}$, we obtain $fQ(f) \in 1+I$ for some $Q \in K[T]$. We then take $g := fQ(f)$. This proves that $D(f)/Z \rightarrow X/Z$ is an open embedding. In conjunction with Lemma \ref{distopencontain}, this also proves (2) for the pair $(X, Z)$, and for opens $Z \subset V \subset X$ such that $V/Z$ exists.

Next we prove that $f$ has the desired universal property with respect to maps into arbitrary $K$-schemes. Given any map $g\colon X \rightarrow W$ as in the proposition, the existence and uniqueness of the map $h$ on $X-Z$ is immediate from (3). Now choose an affine open neighborhood $U \subset W$ of $w$. Then $g^{-1}(U)$ is a neighborhood of $Z$. Choose $Z \subset V \subset g^{-1}(U)$ with $V$ affine open. Then $g|V$ has image landing inside an affine open, hence by the already-treated case of maps between affine schemes, there is a unique map $h_V\colon V/Z \rightarrow W$ factoring $g|V$. Then we may take $h$ to be the glueing of $h_V$ and $g|X-Z$, where we are using the fact that $V/Z \subset X/Z$ is an open embedding. This completes the proof of the existence of $f$ with the required universal property and properties (2)--(7) for affine $X$ (with the assumption that $V/Z$ exists in (2)).

Next we prove the existence of $f$ with the desired universal property in general, as well as the properties (2)--(6). Away from $Z$ the construction is immediate. On the other hand, because $X$ is quasi-projective, there is an affine open subscheme of $X$ containing $Z$, so -- in conjunction with (2) in the affine setting -- the general case reduces to the affine case.
\end{proof}

The following lemma will play a crucial role in the proof of Theorem \ref{nonextaj}.

% Lemma: Existence of suitable rigidification
\begin{lemma}
\label{rigexists}
Let $f\colon X \rightarrow Y$ be a finite birational map of quasi-projective curves over $K$ with schematically dense image. Let $\mathscr{F}$ be the finitely supported sheaf on $Y$ defined by the following exact sequence:
\[
0 \longrightarrow \calO_{Y}^{\times} \longrightarrow f_*(\calO_X^{\times}) \longrightarrow \mathscr{F} \longrightarrow 0.
\]
Let $Z$ be the support of $\mathscr{F}$, and let $U$ be a neighborhood of $Z$, let $u \in \Gamma(f^{-1}(U), \calO_X^{\times})$, and let $\overline{u}$ denote the image of $u$ in ${\rm{H}}^0(Y, \mathscr{F})$. Let $\mathscr{L} \in {\rm{H}}^1(Y, \calO_Y^{\times}) = \Pic(Y)$ be the image of $\overline{u}$ under the connecting map. Finally, let $D \subset Y$ be a finite subscheme with support $Z$. Then there exist a trivialization $\iota\colon \mathscr{L}|_D \xrightarrow{\sim} \calO_D$ of $\mathscr{L}$ along $D$ and an isomorphism $\theta\colon \calO_X \xrightarrow{\sim} f^*(\mathscr{L})$ such that $f^*(\iota)\circ(\theta|_{f^{-1}(D)})\colon \calO_{f^{-1}(D)} \xrightarrow{\sim} \calO_{f^{-1}(D)}$ is multiplication by $u|_{f^{-1}(D)}$.
\end{lemma}

\begin{remark}
The above lemma really depends on some universal choices of sign in order to understand how we identify $\Pic(S)$ with $\check{\rm{H}}^1(S, \Gm)$ for a scheme $S$. Our choices will become clear in the course of the proof of the lemma. Suitably reversing these choices would replace $u$ by $u^{-1}$ in the above. At any rate, this sign issue will not ultimately matter for us.
\end{remark}

\begin{proof}
Consider the open cover $\mathcal{U} := \{Y-Z, U\}$ of $Y$. Then $(\overline{1}, \overline{u})$ defines a \v{C}ech $0$-cocycle in $\check{\rm{C}}^0(\mathcal{U}, \mathscr{F})$ whose image in $\check{\rm{H}}^0$ is $\overline{u}$. Its image in $\check{\rm{H}}^1(\mathcal{U}, \calO_Y^{\times})$ under the connecting map is computed by the usual snake lemma construction. First we lift the cocycle to the element $(1, u) \in \check{\rm{C}}^0(\mathcal{U}, f_*(\calO_X^{\times}))$, and then we form the \v{C}ech differential of this element to obtain the \v{C}ech $1$-cocycle $u \in \check{\rm{C}}^1(\mathcal{U}, \calO_Y^{\times}) \subset \Gamma(U-Z, \calO_Y^{\times})$, and the image of $\overline{u}$ under the connecting map is the image of this element in $\check{H}^1 = \varinjlim_U \check{H}^1(\mathcal{U}, \cdot) = {\rm{H}}^1$. This in turn is the line bundle $\mathscr{L}$ on $Y$ obtained by gluing the trivial line bundles on $Y-Z$ and $U$ along multiplication by $u|U - Z$ on the intersection. (Note: $u$ makes sense on $U-Z$ because $f$ is an isomorphism away from $Z$.) That is, a section of $\mathscr{L}$ over an open $V \subset Y$ is given by a pair $(s, t)$ with $s \in \Gamma(V-Z, \calO_Y)$, $t \in \Gamma(V \cap U, \calO_Y)$ such that $t = us$ on $(V-Z) \cap U$.

Now we take $\theta$ to be the map sending $1 \in \calO_X$ to the generating global section $(1, u)$ of $f^*\mathscr{L}$, where $1 \in \Gamma(X-f^{-1}(Z), \calO_X)$ and $u \in \Gamma(f^{-1}(U), \calO_X)$. Now $\mathscr{L}|_D = (\mathscr{L}|_U)|_D$, and $\mathscr{L}|_U$ is generated by the section $(u^{-1}, 1)$, where $u^{-1} \in \Gamma(U-Z, \calO_Y)$ and $1 \in \Gamma(U, \calO_Y)$. Thus we may define $\iota$ to be the inverse of the map sending $1$ to (the restriction to $D$ of) this section. Now we compute the composition $f^*(\iota)\circ(\theta|_{f^{-1}(D)})$ to be
\[
1 \mapsto (1, u) = u(u^{-1}, 1) \mapsto u,
\]
as required by the lemma.
\end{proof}

A crucial role in the proof of Theorem \ref{nonextaj} will also be played by a certain Albanese type property of generalized Jacobians. Before stating the required result, we recall the definition of generalized (or rigidified) Jacobians. Let $C/K$ be a projective curve, and let $D \subset C$ be a finite subscheme. Then $\Pic_D(C)$ denotes the fppf sheafification of the presheaf sending a $K$-scheme $T$ to the set of all isomorphism classes of pairs $(\mathscr{L}, \theta)$, with $\mathscr{L}$ a line bundle on $C$ and $\theta\colon \mathscr{L} \xrightarrow{\sim} \calO_D$ a trivialization of $\mathscr{L}$ along $D$. The functor $\Pic_D(C)$ admits a natural map to $\Pic_{C/K}$ by simply forgetting the trivialization. This is surjective with kernel isomorphic to the cokernel of the map $\R_{\Gamma/K}(\Gm) \rightarrow \R_{D/K}(\Gm)$, where $\Gamma := {\rm{H}}^0(C, \calO_C)$. In particular, $\Pic_D(C)$ is representable by a smooth $K$-scheme. Then $\Jac_D(C) \subset \Pic_D(C)$ denotes the preimage of $\Jac(C) \subset \Pic_{C/K}$.

Now let $X \subset C^{\rm{sm}}$ be an open subscheme, with irreducible components $X_1, \dots, X_n$. If we are given a collection $\vec{x} = \{x_1, \dots, x_n\}$ of points $x_i \in X_i(K)$, then for any finite subscheme $D \subset C$ with support disjoint from $X$, we obtain a map $\iota_{\vec{x}, D}\colon X \rightarrow \Jac_D(C)$ which intuitively is defined by sending $x \in X$ to $\calO([x]-[x_i])$ for $x_i$ in the component of $x$, together with the canonical trivialization of this line bundle along $D$ which arises from the fact that $D$ and $[x]-[x_i]$ have disjoint support. More precisely, $\iota_{\vec{x},D}$ is the map which on $X_i$ corresponds to the line bundle $\calO([\Delta|_{C \times X_i}] - [x_i \times X_i])$ on $C \times X_i$, together with its canonical trivialization along $D \times X_i$, where $\Delta$ is the diagonal on $C \times C$. Note that this map makes sense because the $X_i$ are disjoint since $X$ is smooth. Note also that $\iota_{\vec{x},D}$ sends each $x_i$ to $0$. The following result, of independent interest, is the Albanese property that we require.

% Proposition: Albanese property of generalized Jacobians
\begin{proposition}
\label{albanese}
Let $C/K$ be a generically smooth projective curve, and let $X \subset C^{\rm{sm}}$ be a dense open subscheme with irreducible components $X_1, \dots, X_n$. Suppose given $\vec{x} = \{x_1, \dots, x_n\}$ with $x_i \in X_i(K)$, and a $K$-scheme map $f\colon X \rightarrow G$ to a commutative $K$-group scheme of finite type. Assume that $f(x_i) = 0$ for $1 \leq i \leq n$. Then there exists a finite subscheme $D \subset C$ with support $C-X$ and a commutative diagram
\[
\begin{tikzcd}
X \arrow{r}{\iota_{\vec{x}},D} \arrow{dr}[swap]{f} & \Jac_D(C) \arrow{d}{\phi} \\
& G
\end{tikzcd}
\]
with $\phi$ a $K$-homomorphism. Furthermore, for each $D$ with support disjoint from $X$, there is at most one such $\phi$.
\end{proposition}

Before proving Proposition \ref{albanese}, we require the following proposition.

% Proposition: Abel-Jacobi map generates the generalized Jacobian
\begin{proposition}
\label{ajgens}
Let $C/K$ be a generically smooth projective curve, let $X \subset C^{\rm{sm}}$ be a dense open subscheme of $C$ with irreducible components $X_1, \dots, X_n$. Let $x_i \in X_i(K)$, and $\vec{x} := \{x_1, \dots, x_n\}$. Let $D \subset C$ be a finite subscheme with ${\rm{supp}}(D) \cap X = \emptyset$. Then the map $$\iota_{\vec{x},D}|_X\colon X \rightarrow \Jac_D(C)$$ generates $\Jac_D(C)$.
\end{proposition}

\begin{proof}
We may assume that $K$ is algebraically closed. First assume that $D = \emptyset$. Every line bundle on $C$ may be written in the form $\calO(E)$ for a divisor $E$ supported on $X$, so $\iota_{\vec{x},D}|_X$ generates $\Jac(C)(K)$, hence generates the smooth group $\Jac(C)$. Now consider the general case. Choose a finite birational morphism of curves $f\colon C \rightarrow C'$ contracting $D$ to a $K$-rational point $c' \in C'(K)$ as in Proposition \ref{contract}. Then we obtain an induced map $f^*\colon \Jac_{c'}(C') \rightarrow \Jac_D(C)$. We claim that this map is surjective. This will suffice, because the natural map $\Jac_{c'}(C') \rightarrow \Jac(C')$ is an isomorphism, since every unit on $c'$ extends to a global unit of $C'$ (and this remains true after base change), hence the proposition will follow from the case $D = \emptyset$.

To prove that $f^*$ is surjective, it suffices to check surjectivity on $K$-points. Lemma \ref{rigexists} implies that the image of $f^*$ contains every $K$-point of $\ker(\Jac_D(C) \rightarrow \Jac(C))$. Thus it only remains to show that the pullback map $\Jac(C')(K) \rightarrow \Jac(C)(K)$ is surjective. For this, it is enough to show that the map $\Pic(C') \rightarrow \Pic(C)$ is surjective. Let $\mathscr{F}$ be the sheaf defined by the following exact sequence
\[
0 \longrightarrow \calO_{C'}^{\times} \longrightarrow f_*(\calO_C^{\times}) \longrightarrow \mathscr{F} \longrightarrow 0.
\]
Then $\mathscr{F}$ is supported at $c'$, hence ${\rm{H}}^1(C', \mathscr{F}) = 0$. It thus only remains to show that the natural map $${\rm{H}}^1(C', f_*(\calO_C^{\times})) \hookrightarrow {\rm{H}}^1(C, \calO_C^{\times})$$ is an equality. For this, in turn, it suffices to show that $\R^1f_*(\calO_C^{\times}) = 0$. This amounts to the assertion that $\Pic(f^{-1}(\calO_{C',c'})) = 0$, which holds because semilocal rings have trivial Picard groups \cite[Tag 02M9]{stacks}.
\end{proof}

\begin{proof}[Proof of Proposition $\ref{albanese}$]
The uniqueness assertion follows from Proposition \ref{ajgens}, so we only need to prove existence. Because $X$ is smooth, we may replace $G$ by its maximal smooth $K$-subgroup scheme (see \cite[Lem.\,C.4.1, Rem.\,C.4.2]{cgp}) and thereby assume that $G$ is smooth. When $C$ is irreducible and geometrically reduced, the proposition follows from \cite[\S10.3, Th.\,2]{neronmodels}. Now consider general $C$ as in the proposition. Let $\pi\colon \widetilde{C} \rightarrow C$ be the map from the normalization of $C_{\rm{red}}$. Then $\widetilde{C} = \coprod_{i=1}^n \overline{X}_i$, where $\overline{X}_i$ is the regular compactification of $X_i$. By the irreducible case applied to each $X_i \subset \overline{X}_i$, there exist finite subschemes $D_i \subset \overline{X}_i$ with support $\overline{X}_i - X_i$ and a commutative diagram
\begin{equation}
\label{albaneseeqn1}
\begin{tikzcd}
X \arrow{r}{\prod\iota_{x_i,D_i}} \arrow{dr}[swap]{f} & \prod_{i=1}^n \Jac_{D_i}(\overline{X}_i) \arrow{d}{\psi} \\
& G
\end{tikzcd}
\end{equation}
for some $K$-homomorphism $\psi$. Now $\pi(D_i) \subset C-X$, and we may choose $D$ with support $C-X$ such that, for all $i$, $D_i \subset \pi^{-1}(D)$. Then we have the pullback map
\[
\pi^*\colon \Jac_D(C) \rightarrow \prod_{i=1}^n \Jac_{D_i}(\overline{X}_i),
\]
and $\pi^*\circ \iota_{\vec{x},D} = \prod_i\iota_{x_i,D_i}$. Thus we may take $\phi := \psi\circ \pi^*$.
\end{proof}

% Section: Proofs of the main results
\section{(Non-)extendability of Abel-Jacobi}

Although the following proposition is not required in order to prove the main theorems, it explains the failure of Theorem \ref{nonextaj} without the assumption of non-modifiability for $C$, and the idea behind the proof is sufficiently related to that of Lemma \ref{existunits} below, which is crucial for the proof of Theorem \ref{nonextaj}, that we include it here.

% Proposition: Modifications don't change Jacobians
\begin{proposition}
\label{modjacobian}
Let $C$ be a geometrically reduced projective curve. Suppose that $C$ is modifiable, and let $f\colon C' \rightarrow C$ be a modification of $C$. Then the pullback map $f^*\colon \Jac(C) \rightarrow \Jac(C')$ is an isomorphism.
\end{proposition}

\begin{proof}
We use the notation of Definition \ref{modifiabledef}. We first note that modifiability is preserved under field extension. Indeed, because $X$ is irreducible and admits a $K$-point, it is geometrically irreducible. Furthermore, if we take two distinct connected components of $C'$ that map to the same component of $C$, this component must be the one containing $x$ (since the modification from $C$ to $C'$ only occurs at the point $x$), hence is geometrically connected (again, because it contains a $K$-point). So the two components still map to the same component upon extending scalars. Thus we may assume that $K$ is algebraically closed.

It suffices to prove that the map $\Pic(C_R) \rightarrow \Pic(C'_R)$ is an isomorphism for every Artin-local $K$-algebra $R$, as then it is an isomorphism on $K$-points, hence surjective with finite connected kernel (because the Jacobians are smooth), and also an isomorphism on $K[\epsilon]/(\epsilon^2)$-points, hence \'etale. Let $\mathscr{G}$ be defined by the following exact sequence:
\[
0 \longrightarrow \calO_{C_R}^{\times} \longrightarrow f_*(\calO_{C'_R}^{\times}) \longrightarrow \mathscr{G} \longrightarrow 0.
\]
Then $\mathscr{G}$ is a skyscraper sheaf supported at $x$. (Note that $C_R$ and $C'_R$ have the same underlying topological spaces as $C$ and $C'$, respectively. Here we use the fact that $K$ is algebraically closed.) We claim that the natural inclusion ${\rm{H}}^1(C_R, f_*(\calO_{C'_R}^{\times})) \hookrightarrow {\rm{H}}^1(C'_R, \calO_{C'_R}^{\times})$ is an isomorphism. This follows from the vanishing of the sheaf $\R^1f_*(\calO_{C'_R}^{\times})$, which in turn follows from the fact that semilocal rings have vanishing Picard groups \cite[Tag 02M9]{stacks}. Because $\mathscr{G}$ is supported on a $0$-dimensional subspace, its higher cohomology vanishes. It thus only remains to show that the map ${\rm{H}}^0(C, f_*(\calO_{C'_R}^{\times})) = {\rm{H}}^0(C', \calO_{C'_R}^{\times}) \rightarrow {\rm{H}}^0(C, \mathscr{G})$ is surjective.

What we must show is that every unit on $C'_R$ in a neighborhood of the fiber above $x$ is (perhaps after shrinking the neighborhood further) the product of (the pullback of) a unit of $C_R$ in a neighborhood of $x$ and a global unit of $C'_R$. Let $Y$ be the union of the irreducible components of $C$ other than $X$. Then, in a neighborhood of the fiber above $x$, $C'_R$ agrees with the disjoint union of $Y_R$ and $X_R$. Because $Y_R$ is a closed subscheme of $C_R$, every section of $\mathscr{G}_x$ is represented by a pair $(u, 1)$, where $u$ is a unit on $X_R$ in a neighborhood of the point $z_R$ above $x_R$, and $1$ is on $Y_R$. Because the map $f$ is not injective on connected components, small neighborhoods of $z_R$ and of the fiber above $x$ in $Y_R$ are contained in distinct connected components of $C'_R$ (because the only topological modification of $C$ to get to $C'$ comes from disconnecting $Y$ from $X$ at $x$). Thus there is a global unit of $C'_R$ which takes the value $1$ on the neighborhood of $Y_R$ near the $x$ fiber and the value $u(z_R) \in R^{\times}$ at $z_R$. Modifying by this unit, we may assume that $u(z_R) = 1$. Now the units $1$ on $Y$ and $u$ near $z_R$ agree on the scheme-theoretic intersection of $X$ and $Y$ near $x$, hence glue to give a unit on $C_R$.
\end{proof}

The following lemma will play a crucial role in the proof of Theorem \ref{nonextaj}.

% Lemma: There exist suitable units
\begin{lemma}
\label{existunits}
Let $C$ be a reduced projective curve, and assume that $C$ is not modifiable and that $K \not\simeq \F_2$ is perfect. Let $\pi\colon \widetilde{C} \rightarrow C$ denote the normalization of $C$ above one singular point $c$ of $C$, and suppose that one has an inclusion of open subschemes $C^{\rm{sm}} \subset C' \subset \widetilde{C}$ with $C'-C^{\mathrm{sm}}$ consisting of a single $K$-point $x$ lying in the fiber above $c$. Let $D \subset C$, $D' \subset \widetilde{C}$ denote finite subschemes with supports $C-C^{\rm{sm}}$, $\widetilde{C}-C'$, respectively. Finally, let $\mathscr{F}$ be the sheaf on $C$ defined by the following exact sequence:
\[
1 \longrightarrow \calO_C^{\times} \longrightarrow \pi_*(\calO_{\widetilde{C}}^{\times}) \longrightarrow \mathscr{F} \longrightarrow 1.
\]
Then there is a unit $u \in \calO_{\widetilde{C}}^{\times}$ defined in a neighborhood of $\pi^{-1}(D)$ such that $u|_{D'} = 1$, and such that the image $\overline{u}$ of $u$ in $\mathscr{F}_c$ does not lift to an element of ${\rm{H}}^0(C, \pi_*(\calO_{\widetilde{C}}^{\times})) = {\rm{H}}^0(\widetilde{C}, \calO_{\widetilde{C}}^{\times})$.
\end{lemma}

\begin{proof}
Suppose first that there is a neighborhood $V$ of $x$ such that $\pi^{-1}(c) \cap V$ equals $\{x\}$ set-theoretically, but not scheme-theoretically. (Equivalently, $\pi^*(\mathfrak{m}_c)$ does not generate $\mathfrak{m}_x$, where these are the maximal ideals of the local rings at the respective points.) Thus we may choose a unit $v$ in a neighborhood of $x$ that is not pulled back from $C$ locally near $x$. Modifying by a scalar, we may further assume that $v(x) = 1$. Because $\pi$ is finite, there is some positive integer $n$ such that $v \pmod{\pi^*(\mathfrak{m}^n)}$ is not pulled back from a unit in $\calO_{C,c}/\mathfrak{m}^n$, where $\mathfrak{m}$ is the maximal ideal of $\calO_{C,c}$. Then we may choose as our $u$ a unit in some neighborhood of $\pi^{-1}(D)$ which takes the value $v$ modulo $(\pi^*(\mathfrak{m}^n))$ at $x$ and takes the value $1$ on $\pi^{-1}(D)-\{x\} \supset D'$. To see that $\overline{u}$ does not come from a global unit on $\widetilde{C}$, we note that ${\rm{H}}^0(\widetilde{C}, \calO_{\widetilde{C}})$ is a product of fields, one for each component of $\widetilde{C}$, because $\widetilde{C}$ is reduced. If $u$ may be written in the form $u_1u_2$ with $u_1$ the restriction of a global unit of $\widetilde{C}$ and $u_2$ a unit of $C$ in a neighborhood of $c$, then we may rescale so that $u_2(c) = 1$, hence also $u_1(x) = 1$. Then $u_1 = 1$ in a neighborhood of $x$, hence we would get that $u$ comes from a unit of $C$, contrary to our construction. This completes the proof in the case in which there is a neighborhood $V$ of $x$ such that $\pi^{-1}(c) \cap V = \{x\}$ set-theoretically but not scheme-theoretically.

Next suppose that $\pi^{-1}(c) \cap V = \{x\}$ scheme-theoretically for some neighborhood $V$ of $x$. Note in particular that we must not have $\pi^{-1}(c) = \{x\}$ set-theoretically, for otherwise, by \cite[Ch.\,II, Lem.\,7.4]{hartshorne}, $\pi$ would be an isomorphism above $c$, in violation of the fact that $c$ is a singular point of $C$. (Singular points are the same as nonsmooth points because $K$ is perfect.) Let $Z$ be the irreducible component of $\widetilde{C}$ containing $x$, and let $X := \pi(Z)$. Also let $Y$ be the union of the other irreducible components of $C$.

Assume that $c \notin X^{\rm{sm}}$. (Of course we have $c \notin C^{\rm{sm}}$, but it may be a smooth point of $X$.) Let $g$ be the restriction of $\pi$ to a map $Z \rightarrow X$, so $g$ is the normalization of $X$ above $c$. Because $X$ is not smooth at $c$, $g$ is not an isomorphism above $c$. Thus, if $g^{-1}(c)$ consists of only the point $x$ set-theoretically, then $g^{-1}(c)$ is strictly larger than $\{x\}$ scheme-theoretically, contrary to our assumption. Therefore, $g^{-1}(c)$ contains more than one point; let $z \neq x$ be another. Then we take $u$ to be $1$ along $D'$ and such that $u(x) = \lambda$ for some $1 \neq \lambda \in K^{\times}$, and we claim that $u$ may not be written locally near $\pi^{-1}(D)$ as a product $u_1u_2$ with $u_1$ a global unit on $\widetilde{C}$ and $u_2$ a unit on $C$ in a neighborhood of $c$. Indeed, we may rescale and thereby assume that $u_2(c) = 1$. Since $u_1$ must be a scalar along each component of $\widetilde{C}$, it follows (because $u(z) = 1$) that $u_1 = 1$ along $Z$. Therefore, $u_2(x) = \lambda$, in violation of the fact that $u_2(c) = 1$. Thus we are done in this case, so may assume that $c$ is a smooth point of $X$.

Choose an open neighborhood $U \subset C$ of $c$ such that $X \cap Y \cap U = \{c\}$ set-theoretically. Suppose that this intersection is strictly larger than $\{c\}$ {\em scheme-theoretically}. Let the intersection be $T$. Then we choose as our unit $u$ something which is $1$ along $D' \cup (\pi^{-1}(T)-\{x\})$, and such that $u(x) = 1$, but $u \neq 1$ in $g^{-1}(T) \subset Z$. Suppose that $u = u_1u_2$ with $u_1$ the restriction of a global unit of $\widetilde{C}$ and $u_2$ a unit on $C$ in a neighborhood of $c$. Rescaling, we may assume that $u_2(c) = 1$. Then $u_1$ -- which is locally scalar -- equals $1$ along $\pi^{-1}(D) \cup \pi^{-1}(T)$. It then follows that $u_2 = 1$ along $T \subset Y$, hence $u = 1$ along $g^{-1}(T)$, contrary to our choice of $u$. Thus the lemma is proven in this case, so we may assume that $X \cap Y \cap U = \{c\}$ scheme-theoretically.

We are therefore now in the situation in which $c \in X^{\rm{sm}}$ and $X \cap Y \cap U = \{c\}$ scheme-theoretically for some neighborhood $U \subset Y$ of $c$. Let $E \subset \widetilde{C}$ be the connected component containing $x$. We claim that $E \cap (D'-\{x\}) \neq \emptyset$. For if not, then consider the curve $C'$ obtained by ``pulling $X$ away from $Y$ along $c$.'' More precisely, $C'$ is obtained by gluing $X-(X\cap Y - \{x\})$ and $C-\{x\}$ along their intersection $X-(X \cap Y)$. If $E \cap (D' - \{x\}) = \emptyset$, then the natural map $f\colon C' \rightarrow C$ is not injective on $\pi_0$, in violation of our assumption that $C$ is not modifiable. Thus $E \cap (D'-\{x\}) \neq \emptyset$; let $y$ be a point of intersection.

Choose $u$ to be $1$ along $D'$ but such that $u(x) = \lambda$, where $1 \neq \lambda \in K^{\times}$. To see that $\overline{u}$ does not lift to an element of ${\rm{H}}^0(\widetilde{C}, \calO_{\widetilde{C}}^{\times})$, note that any such lift has to take the same value at $x$ and $y$. Since the same holds for the pullback of any unit on $C$, and since $u$ takes distinct values at $x$ and $y$, $\overline{u}$ does not lift. This completes the proof of the lemma.
\end{proof}

Next we show that modifiability descends from the algebraic closure.

% Lemma: Modifiability descends
\begin{lemma}
\label{moddesc}
Let $C/K$ be a geometrically reduced projective curve each irreducible component of which is geometrically irreducible. If $C$ is modifiable over $\overline{K}$, then it is so over $K$.
\end{lemma}

\begin{remark}
It is simpler to prove that modifiability descends from the perfect closure, and the only reason that we require the above lemma is to deal with the case $K = \F_2$, for which we need to pass to a larger extension due to the assumption that $K \not\simeq \F_2$ in the proof of Lemma \ref{existunits} (which arose from the need in that proof for there to exist an element $\lambda \in K\backslash\{0, 1\}$). In particular, the reader who does not mind excluding this case is free to just check the part of the proof below which shows that modifiability descends through purely inseparable extensions.
\end{remark}

\begin{remark}
Modifiability does not descend from the algebraic closure in general without the assumption that the irreducible components of $C$ be geometrically irreducible, for rather silly reasons. Indeed, let $L/K$ be a finite Galois extension of degree $d > 1$, and let $C$ be two copies of $\P^1_L$ glued at an $L$-rational ordinary double point, but considered as a curve over $K$. Then $C$ is not modifiable, as it has no $K$-points, but $C_L$, which is just $d$ copies of $C$, now considered as a curve over $L$, is modifiable.
\end{remark}

\begin{proof}
Suppose that one may modify $C_{\overline{K}}$ along $(\overline{X}, \overline{x})$, and let $\overline{Y} \subset C_{\overline{K}}$ denote the union of the irreducible components of $C_{\overline{K}}$ other than $\overline{X}$. Let $\pi\colon C_{\overline{K}} \rightarrow C$ denote the natural map, and let $x := \pi(\overline{x}) \in C$, $X := \pi(\overline{X})$, and let $Y \subset C$ denote the union of the components of $C$ other than $X$. Because there is a neighborhood $\overline{U}$ of $\overline{x}$ such that $\overline{X} \cap \overline{Y} \cap \overline{U} = \{\overline{x}\}$ scheme-theoretically, there is also a neighborhood $U$ of $x$ such that $X \cap Y \cap U = \{x\}$ scheme-theoretically, and this intersection is geometrically reduced, so that $x$ is a separable point of $X$ -- that is, $\kappa(x)$ is a separable extension of $K$. Furthermore, because $\overline{x} \in \overline{X}^{\rm{sm}}$, one also has $x \in X^{\rm{sm}}$.

Next we verify that $x$ is a $K$-point of $C$. Because the irreducible components of $C$ are geometrically irreducible, $\pi$ induces a bijection between irreducible components. If $x$ were not $K$-rational, then it would have $> 1$ preimage under $\pi$, and all of these preimages live in the unique irreducible component $\overline{X}$ above $X$. Thus, the curve $\overline{C}'$ in the definition of modification which is obtained by separating $\overline{X}$ and $\overline{Y}$ along the intersection point $\overline{x}$ has the property that, for $\overline{Y} \subset \overline{C}'$, the components of $\overline{Y}$ containing $\overline{x}$ also intersect $\overline{X} \subset \overline{C}'$ at any of the preimages of $x$ other than $\overline{x}$. In particular, $\overline{C}'$ has the same number of connected components as $C_{\overline{K}}$, since the connected components are determined entirely by the intersection relations among the irreducible components, which are left unchanged. This violates our assumption that $\overline{C}' \rightarrow C_{\overline{K}}$ is a modification. So we conclude that $x$ must be a $K$-rational point.

It only remains to show that separating $X$ and $Y$ at $x$ produces a modification of $C$. This follows from the fact that passing from $C$ to $C_{\overline{K}}$ has no effect on the relation ``intersects nontrivially'' between irreducible components. Because $x$ is $K$-rational, hence has only one preimage, it follows that the equivalence relation between components generated by this relation is the same for the curve $C'$ obtained by separating $X$ and $Y$ at $x$ as for the curve $\overline{C}'$.
\end{proof}

We are ready to prove Theorem \ref{nonextaj}.

\begin{proof}[Proof of Theorem \ref{nonextaj}]
Because each irreducible component of $C$ admits a $K$-point by assumption, each irreducible component is geometrically irreducible. By Lemma \ref{moddesc}, therefore, we may extend scalars and thereby assume that $K$ is algebraically closed. We may assume that $C'-C^{\rm{sm}}$ consists of a single closed point $x$, and we may replace $C'$ by its normalization and thereby assume that $C'$ is smooth. Because $C'$ is separated, it admits a smooth compactification, which is then also the (unique) smooth compactification of $C^{\rm{sm}}$, which is the normalization $C_1$ of $C$. Thus we may assume that $C^{\rm{sm}} \subset C' \subset \widetilde{C}$, where $\widetilde{C}$ is the normalization of $C$ above the image $c$ of $x$ under the normalization map $C_1 \rightarrow C$, and that $C'-C^{\rm{sm}} = \{x\}$.

Assume for the sake of contradiction that the Abel-Jacobi map $\iota_{\vec{c}}\colon C^{\rm{sm}} \rightarrow \Jac(C)$ extends to a map $f\colon C' \rightarrow \Jac(C)$. By Proposition \ref{albanese}, there is a closed subscheme $D' \subset \widetilde{C}$ with support $\widetilde{C} - C'$, and a commutative diagram
\[
\begin{tikzcd}
C' \arrow{r}{\iota_{D', \vec{c}}} \arrow{dr}[swap]{f} & \Jac_{D'}(\widetilde{C}) \arrow{d}{\phi} \\
& \Jac(C)
\end{tikzcd}
\]
with $\phi$ a $K$-homomorphism. Applying the same result again to the map $C^{\rm{sm}} \subset C' \rightarrow \Jac_{D'}(\widetilde{C})$, we obtain a closed subscheme $D \subset C$ with support $C - C^{\rm{sm}}$, and a commutative diagram of $K$-homomorphisms
\begin{equation}
\label{eqn1}
\begin{tikzcd}
\Jac_D(C) \arrow{r}{\psi} \arrow{dr}[swap]{\alpha} & \Jac_{D'}(\widetilde{C}) \arrow{d}{\phi} \\
& \Jac(C)
\end{tikzcd}
\end{equation}
where $\alpha$ is the map which simply forgets the trivialization, by the uniqueness aspect of Proposition \ref{albanese} and the fact that the composition $C \subset C' \xrightarrow{f} \Jac(C)$ is the Abel-Jacobi map $\iota_{C,\vec{c}}$. There is a similar such diagram for any divisor $D_1$ with the same support as $D$ but which contains $D$. Thus we may enlarge $D$ if necessary and thereby assume that $D' \subset \pi^{-1}(D)$, where $\pi\colon \widetilde{C} \rightarrow C$ is the normalization above $c$ map. Then we claim that $\psi$ is the natural map obtained by sending a pair $(\mathscr{L}, \theta)$ to the pair $(\pi^*(\mathscr{L}), \pi^*(\theta)|_{D'})$. To check this, by the uniqueness aspect of Proposition \ref{albanese}, it suffices to check that the two maps agree when restricted to $C$ via the map $\iota_{D, \vec{c}}$. But this follows from the fact that $\psi$ was {\em defined} to be the map associated to the composition $C \subset C' \xrightarrow{\iota_{D',\vec{c}}} \Jac_{D'}(\widetilde{C})$, and this composition indeed pulls back $\calO([c]-[c_i])$ and restricts its canonical trivialization for $c \in C_i$.

We will in fact show that there is no map $\phi$ making diagram (\ref{eqn1}) commute, which will complete the proof. Referring to the exact sequence in the statement of Lemma \ref{existunits}, that lemma says that there is a unit $u \in \calO_{\widetilde{C}}^{\times}$ defined in a neighborhood of $\pi^{-1}(D)$ such that $u|_{D'} = 1$, and such that the image $\overline{u}$ of $u$ in $\mathscr{F}_c$ does not lift to an element of ${\rm{H}}^0(C, \pi_*(\calO_{\widetilde{C}}^{\times}))$. We use this to show the non-existence of $\phi$ as in diagram (\ref{eqn1}). The unit $u$ defines a nontrivial element of ${\rm{H}}^0(C, \mathscr{F})$ which under the connecting map defines a nontrivial line bundle $\mathscr{L}$ on $C$. By Lemma \ref{rigexists} and because $u|_{D'} = 1$, there is a trivialization $\iota\colon\mathscr{L}|_D \xrightarrow{\sim} \calO_D$ such that the pair $(f^*(\mathscr{L}), f^*(\iota)|_{D'})$ is isomorphic to the trivial pair $(\calO_{\widetilde{C}}, {\rm{Id}})$. (Note: Though not the first time in this argument, here we make use of the reducedness of $C$, since we need the map $\pi$ to have image that is scheme-theoretically dense in order to apply Lemma \ref{rigexists}.) The pair $(\mathscr{L}, \iota)$ therefore defines an element of $\ker(\psi)$. But because $\mathscr{L}$ is nontrivial and $K = \overline{K}$, the image of this pair under $\alpha$ is nonzero. It follows that $\ker(\psi) \not\subset \ker(\alpha)$, so there does not exist a map $\phi$ as in (\ref{eqn1}).
\end{proof}

% Section: Abel-Jacobi is a closed embedding for irreducible curves
\section{Abel-Jacobi is a closed embedding for irreducible curves}

In this section we prove Theorem \ref{ajclemb}. The key point is to show that (under the hypotheses of the theorem) the Abel-Jacobi map is a monomorphism. This will be done via a sequence of lemmas. We first note that the result descends from covers.

% Lemma: Monomorphism property descends from covers
\begin{lemma}
\label{desfromcov}
Let $C$ be an irreducible projective curve, $c_0 \in C^{\rm{sm}}(K)$, and let $\pi\colon C' \rightarrow C$ be a dominant morphism from another projective curve such that $\pi$ induces an isomorphism $\pi^{-1}(C^{\rm{sm}}) \xrightarrow{\sim} C^{\rm{sm}}$. If the Abel-Jacobi map $\iota_{C',c_0}\colon C'^{\rm{sm}} \rightarrow \Jac(C')$ is a monomorphism, then so is $\iota_{C,c_0}\colon C^{\rm{sm}} \rightarrow \Jac(C)$.
\end{lemma}

\begin{proof}
This follows immediately from the following commutative diagram:
\[
\begin{tikzcd}
\pi^{-1}(C^{\rm{sm}}) \arrow{r}{\iota_{C',c_0}} \isoarrow{d} \arrow{d}[swap]{\pi} & \Jac(C') \\
C^{\rm{sm}} \arrow{r}{\iota_{c_0}} & \Jac(C) \arrow{u}{\pi^*}
\end{tikzcd}
\]
\end{proof}

Next we show that any non-smooth rational curve admits a cover from either the nodal or cuspidal cubic.

% Lemma: Cover from the nodal and cuspidal cubics
\begin{lemma}
\label{cubiccover}
Let $C$ be an integral, rational, nonsmooth projective curve over an algebraically closed field $K$. There is a finite morphism $\pi\colon D \rightarrow C$ such that $\pi$ induces an isomorphism $\pi^{-1}(C^{\rm{sm}}) \xrightarrow{\sim} C^{\rm{sm}}$, where $D \subset \P^2_K$ is either the nodal cubic curve $\{Y^2Z = XYZ+X^3\}$ or the cuspidal cubic $\{Y^2Z = X^3\}$.
\end{lemma}

\begin{proof}
Let $c_0 \in C(K)$ lie in the singular locus, and let $\widetilde{C} \rightarrow C$ denote the normalization above the complement of $c_0$. Replacing $C$ by $\widetilde{C}$, we may assume that $C-C^{\rm{sm}}$ consists of a single point $c_0$. Let $\pi\colon \P^1 \rightarrow C$ denote the normalization. First assume that $\pi^{-1}(c_0)$ consists set-theoretically of more than one point, and let $a, b \in \P^1(K)$ be distinct points lying above $c_0$. Changing coordinates, we may assume that $a, b = 0, 1$. By Proposition \ref{contract}, $\pi$ factors through a morphism $\P^1/Z \rightarrow C$, where $Z \subset \P^1$ is the reduced subscheme with underlying set $\{0, 1\}$. 

We claim that $\P^1/Z$ is identified with the nodal cubic $D_1$ of the proposition via the map $f\colon \P^1 \rightarrow D_1$ defined on affine charts $\A^1 \rightarrow \{y^2=xy+x^3\}$ by $t \mapsto (t^2-t, t^3-t^2)$. A rational inverse is given by $(x, y) \mapsto y/x$. The above therefore extend to mutually inverse maps above the smooth locus of $D_1$. One readily checks that $D_1$ is smooth away from the point $P := (0, 0)$ of the above affine chart. By Proposition \ref{contract},(7), therefore, in order to conclude that $\P^1/Z = D_1$, we only need to check that $f^*$ identifies the coordinate ring of $\{y^2=xy+x^3\}$ with $K[I] \subset K[t]$, where $I = (t(t-1))$ is the ideal associated to $Z$. Because $D_1$ is reduced and $f$ is dominant, $f^*$ is injective, so we merely need to verify that every element of $I$ lies in the image of $f^*$. That is, we must show that every element of $I$ is a polynomial (with coefficients in $K$) in $t^2-t$ and $t^3-t^2$. Let $F \in I$. We prove the assertion by induction on $d := {\rm{deg}}(F)$. If $F \neq 0$, then $d \geq 2$, so we can write $d = 3m+2n$ for some integers $m,n \geq 0$. Let $c$ denote the leading term of $F(t)$. Then $G(t) := F(t) - c(t^3-t^2)^m(t^2-t)^n$ has degree $< {\rm{deg}}(F)$ and lies in $I$, so by induction $G \in K[t^3-t^2, t^2-t]$, hence the same holds for $F$.

The argument in the case that $\pi^{-1}(c_0)$ consists of a single point set-theoretically is similar. We may assume that this point is $0$. We first must show that $\pi^{-1}(c_0)$ is not the $K$-point $0$ scheme-theoretically. If it were, then \cite[Ch.\,II, Lem.\,7.4]{hartshorne} would imply that $\pi$ is a closed embedding, hence an isomorphism. Thus $V(\mathfrak{m}^2) \subset \pi^{-1}(c_0)$, where $\mathfrak{m}$ is the maximal ideal of the local ring at $0$. Applying Proposition \ref{contract}, we conclude that $\pi$ factors through $\P^1/Z$, where $Z \subset \P^1$ is the closed subscheme defined by $\mathfrak{m}^2$. We therefore need to identify $\P^1/Z$ with the cuspidal cubic $D_2$.

This may be done via the map $g\colon \A^1 \rightarrow \{y^2 = x^3\}$, $t \mapsto (t^2, t^3)$. The above map has rational inverse given by $(x, y) \mapsto y/x$. These maps are therefore mutual inverses above the smooth locus of $D_2$. One readily checks that $D_2$ is smooth away from the point $(0,0)$ in the above affine chart. Thus we only need to verify that $g^*$ identifies the coordinate ring of $\{y^2=x^3\}$ with $K[J] \subset K[t]$, where $J = (t^2)$. Because $g$ is dominant and $D_2$ is reduced, $g^*$ is injective, so we only need to verify that $K[J] = K[t^2, t^3]$. As in the nodal case above, this follows from the fact that every integer $\geq 2$ is of the form $3m+2n$ with $m, n$ nonnegative integers.
\end{proof}

We now check that Theorem \ref{ajclemb} holds for the nodal and cuspidal cubics.

% Lemma: Abel-Jacobi map is an isomorphism for the nodal and cuspidal cubics
\begin{lemma}
\label{ajclembnodecusp}
Let $C$ denote either the nodal or cuspidal cubic of Lemma $\ref{cubiccover}$. Then for any point $c_0 \in C^{\rm{sm}}(K)$, the Abel-Jacobi map $\iota_{C,c_0}$ is an isomorphism.
\end{lemma}

\begin{proof}
We may assume that $K$ is algebraically closed and that $c_0 = [0:0:1]$ because $\mathrm{Aut}(C)$ acts transitively on $C^{\mathrm{sm}}(K)$. We first note that ${\rm{Aut}}(C)$ acts transitively on $C(K)$. Indeed, writing $C = \P^1/Z$ with $Z$ denoting either $\{0, 1\}$ or the the square of the maximal ideal at $0$, one can find automorphisms of $\P^1$ which fix $Z$ and send any $K$-point not in $Z$ to any other such point. Thus we may fix any $K$-point of $C$ that we wish -- say the point $\infty$ at infinity -- and take that as our $c_0$. The assertion of the lemma could be verified by a direct calculation, but we prefer the following approach.

First we check that the Abel-Jacobi map separates points. If not, then there exist two distinct points $x, y \in C(K)$ and a rational function $f$ on $C$, regular and nonvanishing at the singular point, such that ${\rm{div}}(f) = [x]-[y]$. This then defines a morphism $f\colon C \rightarrow \P^1$ of degree $1$. Let $\pi\colon \P^1 \rightarrow C$ denote normalization. Then $f\circ\pi$ is an automorphism of $\P^1$. Letting $h$ denote its inverse and replacing $f$ by $h\circ f$, we obtain a map $f\colon C \rightarrow \P^1$ that is a left inverse to $\pi$. It follows that it is a right inverse on $C^{\rm{sm}}$, hence everywhere. Thus we obtain that $C \simeq \P^1$, a contradiction. So the Abel-Jacobi map does indeed separate points.

Note that the cuspidal and nodal cubics make sense over $\mathrm{Spec}(\Z)$, and we have that $\Jac(C) \simeq C$ as $\Z$-schemes, by the standard calculation of these Jacobians. Let $W$ be a DVR with residue field $K$ and generic characteristic $0$. Then the Abel-Jacobi map $C_W^{\mathrm{sm}} \rightarrow \Jac(C)_W$ is quasi-finite because it separates points, and is an isomorphism on the generic fiber because any homomorphism between connected finite type $L$-groups of the same dimension which separates points must be an isomorphism when $\mathrm{char}(L) = 0$. It follows from Zariski's Main Theorem that the Abel-Jacobi map is an open immersion over all of $W$. It is therefore an isomorphism, since fiberwise it is a homomorphism of connected finite type group schemes.
\end{proof}

We may finally complete the proof of Theorem \ref{ajclemb}.

\begin{proof}[Proof of Theorem $\ref{ajclemb}$]
Note first that $C$ is irreducible and admits a $K$-point, hence is geometrically irreducible. Assume that $C \not\simeq \P^1_K$. If $C_{\overline{K}} \simeq \P^1_{\overline{K}}$, then $C$ is a smooth connected genus $0$ curve with $C(K) \neq \emptyset$, hence $C \simeq \P^1_K$ over $K$. Thus we may assume that $K$ is algebraically closed. We claim that the Abel-Jacobi map associated to $C$ is a monomorphism.

Let $\pi\colon \widetilde{C} \rightarrow C$ be the normalization. If $\widetilde{C} \not\simeq \P^1$, then it is well-known that the Abel-Jacobi map associated to $\widetilde{C}$ is a monomorphism, so the claim holds in this case by Lemma \ref{desfromcov}. Now assume that $\widetilde{C} = \P^1$. By Lemma \ref{cubiccover}, $C$ admits a cover $D \rightarrow C$ which is an isomorphism above the smooth locus of $C$, where $D$ is the cuspidal or nodal cubic. The Abel-Jacobi map associated to $D$ is a monomorphism by Lemma \ref{ajclembnodecusp}, hence also for $C$ by Lemma \ref{desfromcov}. This completes the proof that the Abel-Jacobi map is a monomorphism.

By Corollary \ref{zarclinjac}, the image $C'$ of $\iota_{c_0}$ is closed. Endow it with its reduced structure. We thus obtain a surjective map $f\colon C^{\rm{sm}} \rightarrow C'$ of curves, which we want to be an isomorphism. The induced map $C^{\rm{sm}} \rightarrow \widetilde{C'}$ to the normalization of $C'$ must also be a monomorphism, and by Theorem \ref{nonextaj}, this must be an equality. In particular, $f$ is finite. Since $f$ is a monomorphism, so in particular a homeomorphism, it is an isomorphism by \cite[Ch.\,II, Lem.\,7.4]{hartshorne}.
\end{proof}


\begin{thebibliography}{ram}

\bibitem[BLR]{neronmodels} Siegfried Bosch, Werner L\"utkebohmert, Michel Raynaud, {\em N\'eron Models}, Springer-Verlag, New York, 1990.

\bibitem[CGP]{cgp} Brian Conrad, Ofer Gabber, Gopal Prasad, {\em Pseudo-reductive Groups}, Cambridge Univ.\,Press (2nd edition), 2015.

\bibitem[Fer]{ferrand} Daniel Ferrand, {\em Conducteur, descente et pincement}, Bull.\,Soc.\,Math.\,France 131 (2003), no.\,4, 553--585.

\bibitem[Har]{hartshorne} Robin Hartshorne, {\em Algebraic Geometry}, Springer, 1977.

\bibitem[Over]{overflow} Mathoverflow user Mohan, {\tt{https://mathoverflow.net/questions/344121/\\gluing-two-points-in-an-affine-algebraic-variety}}.

\bibitem[Ros]{rostateduality} Zev Rosengarten, {\em Tate Duality In Positive Dimension Over Function Fields}, Memoirs of the American Mathematical Society, Vol.\,290, No.\,1444, available at {\tt{https://arxiv.org/pdf/1805.00522.pdf}}.

\bibitem[Sch]{schwede} Karl Schwede, {\em Gluing schemes and a scheme without closed points}, Contemp.\,Math.\,, 386, American Mathematical Society, Providence, RI, 2005, 157--172.

\bibitem[Stacks]{stacks} {\em The Stacks Project}, {\tt{https://stacks.math.columbia.edu}}.

\end{thebibliography}
\end{document}